\documentclass[11pt]{article}

\usepackage[margin=2.5cm]{geometry}

\usepackage{mathtools}   
\usepackage{amsfonts,amssymb,amsthm}
\usepackage{mathrsfs}
\usepackage{xfrac}
\usepackage{bbold}

\usepackage{graphicx}
\usepackage{tikz}
\usetikzlibrary{arrows,arrows.meta,automata,positioning,shapes,shapes.geometric,trees,bayesnet}
\usepackage{pgfplots}
\pgfplotsset{compat=1.17}

\usepackage{array}
\usepackage{tabularx}
\usepackage{booktabs}
\usepackage{colortbl}
\usepackage{adjustbox}

\usepackage[table]{xcolor}
\usepackage{soul}
\usepackage[normalem]{ulem} 

\usepackage{algorithm}
\usepackage{algpseudocode}

\usepackage{caption}
\usepackage{authblk}

\usepackage{easyReview}
\usepackage{todonotes}

\usepackage{setspace}
\usepackage{footmisc}

\usepackage{hyperref}
\usepackage{cleveref}

\newtheorem{Theorem}{Theorem}[section]

\newtheorem{Lemma}[Theorem]{Lemma}

\newtheorem{Proposition}[Theorem]{Proposition}
\newtheorem{Definition}[Theorem]{Definition}
\newtheorem{remark}{Remark}[section]
\newtheorem{example}{Example}[section]

\usepackage{stackengine}
\stackMath

\title{Uniform Value and Decidability in \\Ergodic Blind Stochastic Games}

\author{Krishnendu Chatterjee\textsuperscript{1}\hspace*{0.5 cm} David Lurie\textsuperscript{2,3}\hspace*{0.5 cm} Raimundo Saona\textsuperscript{1}\hspace*{0.5 cm} Bruno Ziliotto\textsuperscript{4}}

\date{} 

\begin{document}

\maketitle 

\renewcommand{\thefootnote}{\arabic{footnote}}

\footnotetext[1]{Institute of Science and Technology Austria, Austria.}
\footnotetext[2]{NyxAir, Paris, France.}
\footnotetext[3]{Paris Dauphine University, PSL Research University, Paris, France.}
\footnotetext[4]{Toulouse School of Economics, Université Toulouse Capitole, Institut de Mathématiques de Toulouse, CNRS, France.}

\begin{abstract}
    We study a class of two-player zero-sum stochastic games known as \textit{blind stochastic games}, where players neither observe the state nor receive any information about it during the game. 
    A central concept for analyzing long-duration stochastic games is the \textit{uniform value}. 
    A game has a uniform value $v$ if for every $\varepsilon>0$, Player 1 (resp., Player 2) has a strategy such that, for all sufficiently large $n$, his average payoff over $n$ stages is at least $v-\varepsilon$ (resp., at most $v+\varepsilon$).
    Prior work has shown that the uniform value may not exist in general blind stochastic games. 
    To address this, we introduce a subclass called \textit{ergodic blind stochastic games}, defined by imposing an ergodicity condition on the state transitions. For this subclass, we prove the existence of the uniform value and provide an algorithm to approximate it, establishing the \textit{decidability} of the approximation problem. Notably, this decidability result is novel even in the single-player setting of Partially Observable Markov Decision Processes (POMDPs).
    Furthermore, we show that no algorithm can compute the uniform value exactly, emphasizing the tightness of our result. 
    Finally, we establish that the uniform value is independent of the initial belief.
    
    \paragraph{Keywords.} Stochastic game, uniform value, finite state, decidability, ergodicity, approximation algorithm.
\end{abstract}

\newpage


\section{Introduction}

    Zero-sum stochastic games, introduced by Shapley \cite{shapley1953stochastic}, model the interactions between two opposing players within an environment characterized by finite state and action spaces. At each stage, both players are in a specific state and simultaneously select a pair of public actions, meaning that each player observes the opponent's choice. The combination of the current state and the chosen action pair determines both the stage reward and the transition to the successor state. At each stage of the game, players are aware of the current state as well as the history of past states and actions. Stochastic games extend single-player stochastic games, termed \textit{Markov decision processes} (MDPs) \cite{puterman1994}.
    
    In the $N$-stage game, where $N$ is a positive integer, Player 1 seeks to maximize the expected Cesàro mean of the stage payoffs $\tfrac{1}{N}\sum_{m=1}^NG_m$, while Player 2 aims to minimize it. 
    The $N$-stage game is known to have a value \cite{v1928theorie}, denoted by $v_N$. 
    Stochastic games with long durations can be analyzed using two main approaches:
    \begin{itemize}
        \item 
            The \textit{asymptotic approach} studies the asymptotic behavior of the sequence $(v_N)$ as $N \rightarrow +\infty$. 
            In \cite{bewley1976asymptotic}, Bewley and Kohlberg proved that, in every stochastic game, $(v_N)$ converges pointwise to a limit, denoted by $v$. 
            However, because the optimal strategies of the $N$-stage game often depend heavily on $N$, the pointwise limit of a sequence of these strategies typically performs poorly.
        \item 
            The \textit{uniform approach} considers the existence of strategies which are approximately optimal in all $N$-stage games, provided that $N$ is big enough.
            Mertens and Neyman~\cite{mertens1981stochastic} established that every stochastic game has a \textit{uniform value} denoted $v$, that is, for every $\varepsilon > 0$, Player 1 (resp., Player 2) has a strategy such that, for all sufficiently large $N$, his average payoff over $N$ stages is at least $v-\varepsilon$ (resp., at most $v+\varepsilon$). 
            Additionally, algorithms exist to compute and approximate the uniform value in stochastic games; see~\cite{oliu2021new} for recent examples.
    \end{itemize}

    Unlike standard stochastic games, \textit{blind stochastic games} assume that players receive no information of the current state.
    The only available information consists of the initial probability distribution of the states and the history of past actions. 
    These games represent the simplest model where players have imperfect information about the states, making them an excellent theoretical benchmark. 
    In the single-player case, this model reduces to \textit{blind Markov decision processes} (blind MDPs), which coincide with probabilistic finite automata and has been extensively studied in computer science~\cite{paz1971introduction,rabin1963probabilistic}.
    There are many applications of probabilistic finite automata.
    For example, succinct specification languages for infinite strings as computations \cite{baier2012probabilistic,chatterjee_decidable_2012}, analysis of sequences in computational biology \cite{durbin1998BiologicalSequenceAnalysis}, and speech processing \cite{mohri_finite-state_1997}.
    
    For two-player blind stochastic games, the $N$-stage value is always well-defined \cite{v1928theorie}.
    However, Ziliotto \cite{ziliotto2016zero} provided an example of such games without a uniform value. 
    Motivated by this problem, this paper aims to identify conditions under which the uniform value exists in blind stochastic games, and moreover, can be \textit{computed}. 
    To formalize the computational question, we consider the concept of \textit{decidability}, which addresses the \textit{existence} of algorithms to solve problems, rather than their \textit{efficiency}.

    This paper introduces \textit{ergodic blind stochastic games}, which have been previously studied for the single-player case in the context of automata theory \cite{chadha2013probabilistic}. 
    Intuitively, in ergodic blind stochastic games, the influence of actions taken in the distant past vanishes over time.
    The formal definition is based on the ergodic properties of products of stochastic matrices \cite{kolmogoroff1931analytischen, seneta2006non}. 
    
    \paragraph{Contributions} Our main contributions to ergodic blind stochastic games are the following:
    \begin{itemize}
        \item 
            \textit{Theorem \ref{BMDPdec}} proves the existence of the uniform value as well as the decidability of its approximation. 
        \item 
            \textit{Theorem \ref{EVU}} demonstrates that computing exactly the uniform value is undecidable. 
        \item 
            \textit{Theorem \ref{0opt}} shows that the uniform value is independent of the initial belief.
    \end{itemize}

    \paragraph{Connection to prior work and novelty} 
    Venel \cite{venel2015commutative} proved that the uniform value exists in blind stochastic games where the transition function satisfies a commutative assumption, but did not provide an algorithm to compute it.
    In the simpler single-player case,
    Rosenberg et al. \cite{rosenberg2002blackwell} proved that the uniform value exists in general. 
    However, in this setting, Madani et al. \cite{madani2003undecidability} proved that no algorithm can compute or even approximate the uniform value.
    
    Blind stochastic games are equivalent to stochastic games defined over the set of beliefs, termed \textit{belief stochastic games}. These games belong to the broader class of stochastic games with Borel sets, which are typically studied under ergodic assumptions \cite{hernandez1991recurrence}. However, in this paper, we do not impose assumptions directly on the belief stochastic game; instead, our conditions are on the data of the blind stochastic game. Therefore, one key contribution of the paper is to identify the assumptions on the data of the blind stochastic game that ensure the dynamics of beliefs exhibit ergodic behavior. 
    
    Ergodicity assumptions are commonly used in the stochastic game literature, so it is natural to consider them in the context of blind stochastic games. 
    In fact, authors in \cite{hoffman1966nonterminating, sobel1971noncooperative, vrieze2003stochastic} considered stochastic games with finite states, while Nowak \cite{nowak1999sensitive} studied stochastic games with denumerable state space. 
    Moreover, this hypothesis has also been extensively studied in the single-player case with a finite state space \cite{alden1992rolling, hopp1987new, wachs2011average} and with a denumerable state space \cite{bean1990denumerable, park1993optimal}. 
    Beyond theory, ergodic assumptions in games have also been considered for applications, such as formal analysis of attacks in crypto-currencies \cite{chatterjee2018ergodic}.
    For a detailed survey of ergodicity in MDPs, we refer the reader to Arapostathis et al. \cite{arapostathis1993discrete} and Puterman \cite{puterman1994}. 
    
    Decidability has been extensively studied for various objective functions \cite{chatterjee2007stochastic,chatterjee2016optimal, chatterjee2016decidable,gimbert2014deciding}. Positive results on decidability exist for certain classes of games, under strong assumptions \cite{chatterjee2010probabilistic} and for other objectives \cite{chatterjee2016decidable, chatterjee2013survey, gimbert2014deciding}. 
    Chatterjee et al. \cite{chatterjee2022finite} demonstrated that the uniform value can be guaranteed using finite-memory strategies in the single-player case. However, because there cannot be a computable bound on the memory size, this result does not lead to decidability.
    
    Our approach draws upon a wide range of matrix theory literature \cite{chevalier2017sets,daubechies1992sets,kolmogoroff1931analytischen,seneta2006non,wolfowitz1963products}. 
    To the best of our knowledge, this is the first result leveraging ergodic properties and matrix theory to establish the existence of the uniform value as well as its decidability for a subclass of blind stochastic games. 
    Moreover, Theorem \ref{EVU} underscores the importance of focusing on approximating the uniform value rather than computing the exact uniform value for blind stochastic games. 
    Furthermore, it draws a clear separation between ergodic blind stochastic games and stochastic games, where computing the exact uniform value is decidable \cite{oliu2021new}. 
    Moreover, these results yield two key consequences: first, the decidability (resp., undecidability) of approximating (resp., computing) the uniform value also applies to the single-player case, i.e., the subclass of ergodic blind MDPs; second, our contributions apply to simpler objective functions such as the reachability objective. Indeed, the reachability objective---defined as the probability of reaching a set of target states---can be represented using the long-run average objective \cite{madani2003undecidability}. 

    \paragraph{Outline} Our paper is organized as follows. 
    In Section \ref{MaC}, we introduce the model of blind stochastic games. 
    In Section \ref{WEBMDPs}, we define the class of \textit{ergodic blind stochastic games} and provide conditions under which a blind stochastic game is ergodic. 
    In Section \ref{proofsresults}, we prove our contributions for ergodic blind stochastic games.
    Lastly, Section \ref{ExtPOMDPs} discusses the challenges of extending our approach to hidden stochastic games, where players receive public signals from the current state \cite{renault2020hidden}.
    
\section{Framework}\label{MaC}

    In Section \ref{model}, we present the model of two-player zero-sum blind stochastic games. Subsequently, in Section \ref{CF}, we define the computational problems of interest. Finally, in Section \ref{BRSG}, we demonstrate how any blind stochastic game can be transformed into an equivalent stochastic game with perfect information and a compact state space, called belief stochastic game.
    
    \paragraph{Notation} Sets are represented by calligraphic letters such as $\mathcal{I},\mathcal{J},\mathcal{H},\mathcal{K}$, and $\mathcal{S}$. 
    Elements within these sets are denoted by lowercase letters, such as $i$, $j$, $h$, $k$, and $s$. 
    Random elements are denoted by uppercase letters, such as $I$, $J$, $H$, $K$, and $S$. 
    For a finite set $\mathcal{C}$, let $\Delta(\mathcal{C})$ be the set of probability distributions over $\mathcal{C}$, and let $\delta_c$ be the Dirac measure at some element $c\in \mathcal{C}$. 
    For integers $a \leq b$, let $[a..b]$ be the set of integers $\{ a, a+1, \ldots, b \}$. The set of real numbers is denoted by $\mathbb{R}$, and the sets of natural numbers and nonzero natural numbers are denoted by $\mathbb{N}$ and $\mathbb{N}^\ast$, respectively. 

    \subsection{Model Description}\label{model}
        
        \paragraph{Framework} A \textit{two-player zero-sum blind stochastic game}, denoted by $\Gamma$, is defined by a $5$-tuple $\Gamma=(\mathcal{K},\mathcal{I},\mathcal{J},$ $p,g)$ where:
        \begin{itemize}
            \item $\mathcal{K}$ is the finite set of states;
            \item $\mathcal{I}$ and $\mathcal{J}$ are the finite action sets of Player 1 and Player 2, respectively;
            \item $p \colon \mathcal{K}\times \mathcal{I}\times\mathcal{J}\rightarrow\Delta(\mathcal{K})$, represented as $p(k^\prime|k,i,j) \coloneqq p(k,i,j)(k^\prime)$, is the probabilistic transition function that gives the probability distribution over the successor states given a state $k\in \mathcal{K}$ and an action pair $(i,j)\in \mathcal{I}\times\mathcal{J}$. We represent by $P(i,j)$ the transition matrix for every action pair $(i,j)\in\mathcal{I}\times \mathcal{J}$;
            \item $g \colon \mathcal{K}\times\mathcal{I}\times \mathcal{J}\rightarrow[0,1]$ is the stage reward function.
        \end{itemize}

        \noindent In contrast to blind stochastic games, (standard) stochastic games \cite{shapley1953stochastic} feature complete observation of the state variable. 

        \paragraph{Outline of the Game} Let $b_1\in\Delta(\mathcal{K})$ be the initial belief. The blind stochastic game starting from $b_1$, denoted by $\Gamma(b_1)$, evolves as follows:
        \begin{itemize}
            \item 
                An initial state $K_1$ is selected according to $b_1$. 
                The players know $b_1$, but do not know $k_1$, the realization of $K_1$.
            \item 
                At each stage $m\geq 1$, Player 1 and Player 2 select simultaneously actions $I_m$ and $J_m$, respectively. 
                This action pair results in the unobserved stage reward $G_m\coloneqq g(K_m,I_m,J_m)$. Subsequently, the next state $K_{m+1}$ is determined according to the transition probability function $p(K_m,I_m,J_m)$. 
                Players receive no information about the environment. Therefore, they do not observe the state $K_{m+1}$ or the reward $G_m$.
        \end{itemize}
    

        \paragraph{Strategies} In a blind stochastic game, at each stage $m \geq 1$, each player remembers both its own and its opponent's past actions, constituting the \textit{history before stage m}. Formally, let $\mathcal{H}_m\coloneqq (\mathcal{I}\times \mathcal{J})^{m-1}$ define the set of histories before stage $m$, with $(\mathcal{I}\times\mathcal{J})^{0}\coloneqq\{\emptyset\}$. A strategy of Player $1$ is a mapping $\sigma\colon\bigcup_{m\geq 1}\mathcal{H}_m\to \Delta(\mathcal{I})$, and a strategy of Player $2$ is a mapping $\pi\colon \bigcup_{m\geq 1}\mathcal{H}_m\to \Delta(\mathcal{J})$. We denote by $\Sigma$ and $\Pi$ the players' respective strategy sets. For $m\geq 1$, we use the notation $\sigma(i|h_m)$ to represent the probability of selecting action $i\in \mathcal{I}$ given the history $h_m\in \mathcal{H}_m$. Similarly, $\pi(j|h_m)$ denotes the probability of taking action $j\in \mathcal{J}$ given the history $h_m\in \mathcal{H}_m$.
            
        \paragraph{Values} We define the classical notion of value for blind stochastic games as follows. Let $\mathbb{P}_{\sigma,\pi}^{b_1}$ be the law induced by the pair of strategies $(\sigma,\pi)$ and the initial belief $b_1$ on the set of plays of the game $\Omega=(\mathcal{K}\times\mathcal{I}\times\mathcal{J})^\mathbb{N}$. Similarly, $\mathbb{E}^{b_1}_{\sigma,\pi}$ represents the expectation with respect to this law. 

        Let $N\in \mathbb{N}^\ast$. The $N$-stage objective of the blind stochastic game given by strategy pair $(\sigma,\pi)$ is defined by
        \begin{equation*}
            \gamma_N^{b_1}(\sigma,\pi)\coloneqq \mathbb{E}_{\sigma,\pi}^{b_1}\left(\dfrac{1}{N}\sum_{m=1}^N G_m\right).
        \end{equation*} 
        The $N$-stage game has a value \cite{v1928theorie}, denoted by $v_N(b_1)$:
        \begin{equation*}
            v_N(b_1)\coloneqq\max_{\sigma\in \Sigma}\min_{\pi\in \Pi}\gamma_N^{b_1}(\sigma,\pi)=\min_{\pi\in \Pi}\max_{\sigma\in \Sigma}\gamma_N^{b_1}(\sigma,\pi).
        \end{equation*}

        \noindent We now define the uniform value in blind stochastic games.
        
        \begin{Definition}[Uniform Value]
            The blind stochastic game $\Gamma$ has a \textit{uniform value} $v\colon\Delta(\mathcal{K})\to [0,1]$ if, for all $b_1\in\Delta(\mathcal{K})$, for all $\varepsilon>0$, there exists $(\sigma^\ast,\pi^\ast)\in \Sigma\times \Pi$ and $\overline{n}\in \mathbb{N}^\ast$, such that for all $N\geq\overline{n}$ and $(\sigma,\pi)\in \Sigma\times \Pi$, we have
            \begin{align*}
                \gamma_N^{b_1}(\sigma^\ast,\pi)\geq v(b_1) - \varepsilon,
            \end{align*}
            and
            \begin{align*}
                \gamma_N^{b_1}(\sigma,\pi^\ast)\leq v(b_1) + \varepsilon.
            \end{align*}
        \end{Definition}

        \noindent The existence of the uniform value implies that $(v_N)$ converges as $N \rightarrow +\infty$. 
        In (standard) stochastic games \cite{shapley1953stochastic}, Mertens and Neyman \cite{mertens1981stochastic} have proved that the uniform value exists, but it might not be the case in blind stochastic games \cite{ziliotto2016zero}. 

    \subsection{Computational Formalism}\label{CF}

        
        A \textit{decision problem} is a binary question that determines whether a specific property holds for a given input. It is said \textit{decidable} if there exists an algorithm that can solve it for all inputs and \textit{undecidable} otherwise. \\
        
        The decision problem of \textit{computing} the uniform value is defined as follows.

        \begin{Definition}[Decision Version of Computing the Uniform Value]
            Let $\Gamma$ be a blind stochastic game with a uniform value and $b_1\in\Delta(\mathcal{K})$. Given $x\in[0,1]$, the problem asks whether $v(b_1)>x$ or $v(b_1)\leq x$ holds.
        \end{Definition}

        In many real-world scenarios, finding exact solutions can be very challenging. 
        To address this issue, approximation algorithms offer practical and efficient solutions by focusing on $\varepsilon$-optimal solutions.
        The study of the approximability of decision problems involves analyzing the trade-off between the quality of the solution and the computational resources. The decision problem of \textit{approximating} the uniform value is defined as follows.
        
        \begin{Definition}[Decision Version of Approximating the Uniform Value]
            Let $\Gamma$ be a blind stochastic game with a uniform value and $b_1 \in \Delta(\mathcal{K})$. 
            Given $x\in[0,1]$ and $\varepsilon>0$, the problem asks whether $v(b_1)>x+\varepsilon$ or $v(b_1)<x-\varepsilon$ holds.
            Otherwise, the problem can either accept or reject.
        \end{Definition}
        
        In \cite{madani2003undecidability}, Madani et al. established the undecidability of computing and approximating the uniform value for blind MDPs. 
        As all negative results apply to the more general class, these undecidability results also apply to blind stochastic games. As a consequence, it becomes essential to propose conditions that identify decidable subclasses of blind stochastic games. 
    
    \subsection{From Blind to Belief Stochastic Games}\label{BRSG}

        In blind stochastic games, players do not observe the current state, while in stochastic games, players have complete information of the state at the beginning of each stage. First, we introduce the definition of history before stage $m$ as follows.
        
        \begin{Definition}[$m$-Stage History] 
            Given a strategy pair $(\sigma,\pi)\in\Sigma\times \Pi$ and an initial belief $b_1\in\Delta(\mathcal{K})$, denote the (random) history at stage $m$ by
            \begin{equation*}
                    H_m\coloneqq(I_1,J_1,I_2,J_2\ldots,I_{m-1},J_{m-1}).
            \end{equation*}
            The random variable $H_m$ takes values in $\mathcal{H}_m=(\mathcal{I}\times\mathcal{J})^{m-1}$.
        \end{Definition}
            
        In blind stochastic games, players remember the history of past actions when deciding on a new action. Unfortunately, the representation of past histories is complex.
        Indeed, the set of possible histories up to stage $m$ grows exponentially with $m$. 
        An alternative approach summarizes all the information from past actions into a probability distribution known as the \textit{belief} over the state space $\mathcal{K}$ (see Mertens et al. \cite{mertens2015repeated}). 
        Formally, we define the $m$-stage belief in blind stochastic games as follows.

        \begin{Definition}[$m$-Stage Belief]
            Given an initial belief $b_1\in\Delta(\mathcal{K})$, a pair of strategies $(\sigma,\pi)\in\Sigma\times \Pi$, and a history $h_m\in \mathcal{H}_m$, the belief at stage $m$ given history $h_m$ is defined by            \begin{equation*}
                b_{m,\sigma,\pi}^{b_1,h_m}(k)\coloneqq\mathbb{P}_{\sigma,\pi}^{b_1}(K_m=k|H_m=h_m).
            \end{equation*}
        \end{Definition}
        
        \noindent When clear from context, we will simplify the notation of the belief as $b_m$, omitting its dependence on $b_1$, $h_m$, $\sigma$, and $\pi$.
        Given a fixed pair of strategies $\sigma$ and $\pi$, a history $h_m$, and an initial belief $b_1$, one can use Bayes' rule to compute $b_m$. 

        \textrm{}
        A common approach to analyzing blind stochastic games involves considering an auxiliary game in which the state corresponds to the belief of the original game and is perfectly observed. Formally, let $b_1\in \Delta(\mathcal{K})$ be an initial belief and consider a blind stochastic game $\Gamma=(\mathcal{K},\mathcal{I},\mathcal{J},p,g)$. The belief stochastic game, denoted by $\mathcal{G}$, is defined by a $5$-tuple $\mathcal{G}=(\Delta(\mathcal{K}),\mathcal{I},\mathcal{J},\overline{p},\overline{g})$, where:
            \begin{itemize}
                \item $\Delta(\mathcal{K})$ is the infinite set of belief states;
                \item $\mathcal{I}$ and $\mathcal{J}$ are the finite sets of actions of Player 1 and Player 2, respectively;
                \item $\overline{p}\colon\Delta(\mathcal{K})\times \mathcal{I}\times \mathcal{J}\rightarrow \Delta(\mathcal{K})$, denoted $\overline{p}(b^\prime|b,i,j)$, is the deterministic transition function that gives the successor belief states given the current belief state $b\in \Delta(\mathcal{K})$ and the action pair $(i,j)\in \mathcal{I}\times\mathcal{J}$;
                \item $\overline{g}\colon\Delta(\mathcal{K})\times\mathcal{I}\times \mathcal{J}\rightarrow[0,1]$ is the stage reward function, defined for all $b\in\Delta(\mathcal{K})$ and $(i,j)\in\mathcal{I}\times\mathcal{J}$ by $\overline{g}(b,i,j)\coloneqq\sum_{k\in\mathcal{K}} b(k)g(k,i,j)$.
            \end{itemize}

        \noindent At each stage, the belief changes according to the action pair taken. For each stage $m\geq 1$ and state $k^\prime\in \mathcal{K}$, the \textit{belief update} is defined as
        \begin{align*}
            b_{m+1}(k^\prime)\coloneqq& \psi(b_{m},i_{m},j_m)\\
                \coloneqq&\sum_{k\in\mathcal{K}}p(k^\prime|k,i_{m},j_m)b_{m}(k),
        \end{align*}
        and the \textit{deterministic transition function} is defined as
        \begin{align*}
            \overline{p}(b_{m+1}|b_m,i_m,j_m)\coloneqq  
            \left\{
            \begin{array}{ll}
                1 & \mbox{if } b_{m+1}=\psi(b_{m},i_{m},j_m) \\
                0 & \mbox{otherwise.}
            \end{array}
            \right.
        \end{align*}
        
        For each stage $m\geq 1$, let $\overline{G}_m\coloneqq \overline{g}(B_m,I_m,J_m)$ denote the stage reward where $B_m\in\Delta(\mathcal{K})$ and $(I_m,J_m)\in\mathcal{I}\times \mathcal{J}$. Let $N\in \mathbb{N}^\ast$. The $N$-stage objective of the belief stochastic game given by strategy pair $(\sigma,\pi)$ is defined by
        \begin{equation*}
            \gamma_N^{b_1,\prime}(\sigma,\pi)\coloneqq \mathbb{E}_{\sigma,\pi}^{b_1}\left(\dfrac{1}{N}\sum_{m=1}^N\overline{G}_m\right),
        \end{equation*} 
        and the $N$-stage value, denoted by $v_N^\prime(b_1)$, is defined by
        \begin{equation*}
            v_N^\prime(b_1)\coloneqq\sup_{\sigma\in \Sigma}\inf_{\pi\in \Pi}\gamma_N^{b_1,\prime}(\sigma,\pi)=\inf_{\pi\in \Pi}\sup_{\sigma\in \Sigma}\gamma_N^{b_1,\prime}(\sigma,\pi).
        \end{equation*}

        The uniform value of the belief stochastic game, when it exists, is equal to the uniform value of the corresponding blind stochastic game \cite{mertens2015repeated}. The advantage of considering the belief stochastic game is that it can be analyzed using the tools developed for stochastic games with observed states. However, a significant drawback is that the state space becomes infinite. As a result, the existence theorem of Mertens and Neyman \cite{mertens1981stochastic} for the uniform value cannot be directly applied.
    
\section{Ergodic Blind Stochastic Games}\label{WEBMDPs}

    In Section \ref{CD}, we define the class of \textit{ergodic blind stochastic games}. Section \ref{SC} establishes sufficient conditions under which a blind stochastic game is ergodic. Our main contributions to the study of ergodic blind stochastic games are detailed in Section \ref{MR}. Finally, Section \ref{VerErgProp} demonstrates that the problem of determining whether a blind stochastic game is ergodic is decidable.
    
    \subsection{Class Description}\label{CD}
                
        Let $n\geq 1$ and consider an action pair sequence $a^n=(a_1,...,a_n)\in (\mathcal{I}\times \mathcal{J})^n$, where each action pair $a_m=(i_m,j_m)\in\mathcal{I}\times\mathcal{J}$ for all $m\in [1..n]$. We define the forward products of transition matrices by
        \begin{equation*}
            T^{n}(a^n)
                \coloneqq P(a_1)P(a_{2})\cdots P(a_n).
        \end{equation*}
        We say that $P>0$ if $p_{k,k^\prime}>0$ for each $k,k^\prime\in\mathcal{K}$. 
        Similarly, we write $P\geq 0$ if $p_{k,k^\prime}\geq 0$ for each $k,k^\prime\in\mathcal{K}$. 
        When it is clear from the context, we denote $T^{n}(a^n) = T^{n} = (t^n_{k, k'})_{k, k' \in \mathcal{K}}$. 
        Given a matrix $P$, we represent its $k$-th column as $(P)_{k}$ for each $k\in\mathcal{K}$.
        Finally, the transpose of a vector $p$ will be denoted by $p^\top$. A square matrix $P$ is stochastic if $p_{k,k^\prime}\geq 0$ for each $k,k^\prime\in \mathcal{K}$, and the terms of each row sum to one, i.e., $\sum_{k^\prime\in\mathcal{K}}p_{k,k^\prime}=1$ for each $k\in\mathcal{K}$. 
        
        The ergodicity of products of stochastic matrices is defined as follows \cite[Definition 4.4, p. 136]{seneta2006non}.
        
        \begin{Definition}[Ergodicity]\label{PWTCVG}
            A sequence of stochastic matrices $\{ P_i \}_{i\geq 1}$ on $\mathcal{K}\times\mathcal{K}$ is ergodic if we have that, for all $k, \overline{k}, k^\prime \in \mathcal{K}$, 
            \begin{equation*}
                \lim_{n\to\infty} t^{n}_{k,k^\prime}-t^{n}_{\overline{k},k^\prime} = 0. \label{WEC}
            \end{equation*}
        \end{Definition}
    
        \begin{remark}
            In the literature, Definition \ref{PWTCVG} is also commonly known as the weak ergodicity property of forward products of stochastic matrices \cite{seneta2006non}. Moreover, the ``strong" ergodicity of products of stochastic matrices \cite[Definition 4.5, p. 136]{seneta2006non} requires entrywise convergence, that is, each term $t^n_{k,k^\prime}$ must converge to a limit as $n\to\infty$. In contrast, Definition \ref{PWTCVG} emphasizes a convergence based on the differences between rows. Specifically, it indicates that the values of $t_{k,k^\prime}^{n}$ for each $k,k^\prime\in\mathcal{K}$ may not necessarily converge to a limit as $n\to\infty$.
        \end{remark}
        
        Coefficients of ergodicity have been introduced as tools to characterize the convergence speed of forward products of matrices. For a deeper dive into this topic, we refer the reader to the following papers \cite{iosifescu1972two, mott1957xxiv, seneta2006non}. A stochastic matrix $P$ is called \textit{stable} if every row is identical. We define coefficients of ergodicity for stochastic matrices \cite[Definition 4.6, p. 136]{seneta2006non} as follows.
        
        \begin{Definition}[Coefficient of Ergodicity]
            Let $b\in \Delta(\mathcal{K})$ and $\textnormal{\textbf{1}}=(1,\dots,1)\in\mathbb{R}^{|\mathcal{K}|}$. A scalar function $\tau_1(\cdot)$, continuous on the set of $|\mathcal{K}|\times |\mathcal{K}|$ stochastic matrices and satisfying $0\leq\tau_1(P)\leq 1$, is called a coefficient of ergodicity. Moreover, the function $\tau_1$ is proper if 
            \begin{equation*}
                \tau_1(P)=0 \textrm{ if and only if }P=\textnormal{\textbf{1}}b^\top,
            \end{equation*}
            that is, whenever $P$ is stable.
        \end{Definition}
        
        \noindent Given a stochastic matrix $P$, an example of proper coefficient of ergodicity (see \cite{seneta2006non}), denoted by $\tau_1$, is defined by
        \begin{equation*}
            \tau_1(P)\coloneqq\dfrac{1}{2}\max_{k,\overline{k}}\sum_{k^\prime=1}^{|\mathcal{K}|}\left|p_{k,k^\prime}-p_{\overline{k},k^\prime}\right|.
        \end{equation*}
        Using Seneta \cite[Lemma 4.3, p. 139]{seneta2006non}, the coefficient of ergodicity $\tau_1$ is submultiplicative, i.e., for all stochastic matrices $P$ and $Q$, we have that $\tau_1(PQ)\leq\tau_1(P)\tau_1(Q)$. The coefficient of ergodicity $\tau_1$ plays a crucial role in characterizing ergodicity \cite[p. 140]{seneta2006non}.  More specifically, by \cite[Lemma 4.1, p. 136]{seneta2006non}, the ergodicity of forward products of stochastic matrices is equivalent to 
        \begin{equation*}
            \lim_{n\to\infty}\tau_1(T^{n})=0.
        \end{equation*}
        
        We now define the class of \textit{ergodic blind stochastic games}.
        
        \begin{Definition}[Ergodic blind stochastic game]\label{WEUMDP}
            A blind stochastic game $\Gamma$ is ergodic if, for all $\varepsilon>0$, there exists an integer $n_0$ such that, for all action pair sequences $a^n$ with $n\geq n_0$, 
            \begin{equation}
                \tau_1(T^n(a^n))\leq \varepsilon \label{ineqref}.
            \end{equation}
        \end{Definition}
        
        \begin{remark}
            Definition \ref{WEUMDP} considers a ``uniform" $n_0$, i.e., inequality \eqref{ineqref} holds for all action pair sequences $a^{n_0}$. Definition \ref{PWTCVG} characterizes the ergodic property using pointwise convergence. 
            Similarly, we could call a blind stochastic game $\Gamma$ \textit{pointwise ergodic} if, for each action pair sequence, the forward products of transition matrices satisfy the ergodicity condition. 
            By \cite[Theorem 6.1]{daubechies1992sets}, this is equivalent to Definition \ref{WEUMDP}. 
        \end{remark}

    \subsection{Sufficient Conditions}\label{SC}

        We establish that determining the ergodicity of blind stochastic games depends on specific properties of their transition matrices.
        Previous research, see \cite{seneta2006non} for an extensive survey, has identified certain classes of stochastic matrices that ensure ergodicity. 
        To formalize this, let $P$ be a stochastic matrix and let $\mathcal{Q}\subseteq\mathcal{K}$. 
        We define the reachability function 
        \[
            F_P(\mathcal{Q})=\{k^\prime\in\mathcal{K}\, | \,\exists k\in \mathcal{Q}\textrm{ }s.t.\textrm{ }p_{k,k^\prime}>0\},
        \]
        which collects all the states that can be reached from $\mathcal{Q}$ in a single step. 
        We begin by defining these matrix classes and then draw connections between them and subclasses of ergodic blind stochastic games.

        \begin{Definition}[\cite{chevalier2017sets, paz1971introduction, seneta2006non,wolfowitz1963products}]\label{DefClassMatrices}
            \textrm{}
            \begin{itemize}
                \item A matrix $P$ is stochastic indecomposable and aperiodic (SIA), if $\lim_{n\to\infty}P^n=Q$ exists, where $Q$ is a stable stochastic matrix. Denote by $\mathcal{C}_1$ the class of SIA matrices.
                
                \item A stochastic matrix P is a $\mathcal{C}_2$-matrix if $P \in \mathcal{C}_1$ and, for all $Q\in \mathcal{C}_1$, we have that $QP \in \mathcal{C}_1$. Denote by $\mathcal{C}_2$ the class of $\mathcal{C}_2$-matrices.
                    
                \item  A stochastic matrix $P$ is Sarymsakov if, for all two nonempty disjoint subsets $\mathcal{Q},\mathcal{Q}^\prime\subseteq\mathcal{K}$, either there exists a state that can be reached from both $\mathcal{Q}$ and $\mathcal{Q}^\prime$, or the set of reachable states from $\mathcal{Q} \cup \mathcal{Q}^\prime$ has more elements than $\mathcal{Q} \cup \mathcal{Q}^\prime$.
                Formally, a stochastic matrix $P$ is a Sarymsakov matrix if for all two nonempty disjoint subsets $\mathcal{Q},\mathcal{Q}^\prime\subseteq\mathcal{K}$, $F_P(\mathcal{Q})\cap F_P(\mathcal{Q}^\prime)\neq \emptyset$ or $\left|F_P(\mathcal{Q})\cup F_P(\mathcal{Q}^\prime)\right|>\left|\mathcal{Q}\cup \mathcal{Q}^\prime\right|$. 
                Denote by $\mathcal{C}_3$ the class of Sarymsakov matrices.

                \item A stochastic matrix $P$ is scrambling if given any two rows $k$ and $\overline{k}$, there is at least one column $k^\prime$ such that $p_{k,k^\prime}> 0$ and $p_{\overline{k},k^\prime}>0$. Denote by $\mathcal{C}_4$ the class of scrambling matrices.
                    
                \item A stochastic matrix $P$ is Markov if at least one column of $P$ has all entries strictly positive. Denote by $\mathcal{C}_5$ the class of Markov matrices.
            \end{itemize}
        \end{Definition}
        
        A blind stochastic game satisfies the Wolfowitz condition \cite{paz1971introduction} if, for all $n \geq 1$ and any sequence of action pairs $a^n$, the matrix $T^n(a^n)$ belongs to the matrix class $\mathcal{C}_1$. 
        This subclass of blind stochastic games is ergodic. 
        Specifically, there exists an integer $n_0$ such that, for all action pair sequences $a^{n_0}$, the inequality $\tau_1\left(T^{n_0}(a^{n_0})\right)<1$ holds \cite{paz1965definite}. 
        Using Paz \cite[Theorem 3.1, p. 80]{paz1971introduction}, the ergodicity of blind stochastic games satisfying the Wolfowitz condition follows. 
        Moreover, by Seneta \cite{seneta2006non}, the matrix classes satisfy that $\mathcal{C}_5 \subsetneq \mathcal{C}_4 \subsetneq \mathcal{C}_3 \subsetneq \mathcal{C}_2 \subsetneq \mathcal{C}_1$.
        Therefore, by definition of $\mathcal{C}_2$, it follows that every blind stochastic game with transition matrices belonging to the classes $\mathcal{C}_5,\mathcal{C}_4,\mathcal{C}_3$ or $\mathcal{C}_2$ is also ergodic.

    \subsection{Statement of the Results}\label{MR}
    
        We now present our main results, with the proofs postponed to Section \ref{proofsresults}. 
        
        \begin{Theorem}\label{BMDPdec}
            All ergodic blind stochastic games have a uniform value. Moreover, the decision version of approximating the uniform value for the class of ergodic blind stochastic games is decidable.
        \end{Theorem}
        
        \noindent Our approach reduces an ergodic blind stochastic game to a (standard) stochastic game with a double-exponential number of states. 
        Because such games can be solved in \textnormal{PSPACE} via the theory of reals-closed fields \cite{chatterjee2008stochastic}, we obtain a 2-\textnormal{EXPSPACE} upper bound as computational complexity.\\

        Define a blind MDP as Markov if $P(i)\in \mathcal{C}_5$ for every $i\in\mathcal{I}$. 

        \begin{Theorem}\label{EVU}
            The decision version of computing the uniform value for the class of Markov blind MDPs is undecidable. In particular, computing the uniform value for ergodic blind stochastic games is undecidable.
        \end{Theorem}
    
        Theorem \ref{EVU} highlights that the decidability of the approximation problem is a ``tight'' result, as it cannot be extended to the exact problem. Moreover, it establishes a separation between standard stochastic games, where the exact problem is decidable, and blind stochastic games, where it is not.
            
        \begin{Theorem}\label{0opt}
            For every ergodic blind stochastic game, the uniform value is independent of the initial belief.
        \end{Theorem}

    \subsection{Verifying the Ergodic Property}\label{VerErgProp}
        
        We consider the problem of deciding whether the ergodic property holds for a given blind stochastic game. 
        We show that this problem is decidable and that it can be done within exponential space. 
        From Paz \cite[Corollary 4.6 and Theorem 4.7, p. 90]{paz1971introduction}, we have the following proposition.
    
        \begin{Proposition}
        \label{UB}
            A blind stochastic game $\Gamma$ is ergodic if and only if there exists an integer $n_0 \leq \tfrac{3^{|\mathcal{K}|}-2^{|\mathcal{K}|+1}+1}{2}$ such that, for every action pair sequence $a^n$ with $n\geq n_0$,
            \begin{equation}
                \tau_1(T^n(a^n)) < 1 . \label{check}
            \end{equation}
        \end{Proposition}
        
        \noindent Proposition \ref{UB} will form the basis for the proof of Theorem \ref{BMDPdec} in Section \ref{Dec}. Building on Proposition \ref{UB}, we now consider the following result.

        \begin{Proposition}\label{veryergo}
            Let $\Gamma$ be a blind stochastic game. Verifying whether the ergodic property holds for $\Gamma$ is decidable using exponential space.
        \end{Proposition}

        \begin{proof}[Proof of Proposition \ref{veryergo}]
            Although Proposition \ref{UB} states a condition for all sequences of action pairs $a^n$ with $n \ge n_0$, it is sufficient that the condition holds for sequences of action pairs of length $n_0$ only. Therefore, Proposition \ref{UB} immediately implies an algorithm to decide whether a blind stochastic game satisfies the ergodic property. 
            Indeed, it is sufficient to verify if there is $n_0 \leq \left( 3^{|\mathcal{K}|}-2^{|\mathcal{K}|+1}+1 \right) / 2$ such that, for all sequences of action pairs of length $n_0$, we have that $\tau_1(T^{n_0}(a^{n_0})) < 1$.
            Proceeding by enumeration, we can check whether a blind stochastic game satisfies the ergodic property in exponential space because we verify through enumeration whether $\tau_1(T^n(a^n))<1$ is satisfied for every sequence of action pairs of size $n\leq (3^{|\mathcal{K}|}-2^{|\mathcal{K}|+1}+1)/2$.
        \end{proof}

        We also present a simple example, namely, \Cref{exampleBMDP}, which illustrates an application to a machine maintenance problem as an ergodic blind MDP.

        \begin{example}\label{exampleBMDP}
            A player monitors an inaccessible machine, which can be in one of the three following states: $\mathcal{K}=\{\textrm{Good Condition}, \textrm{Fair Condition},\textrm{Poor Condition}\}$. 
            They can take one of the following actions: $\mathcal{I}=\{\textrm{Wait}, \textrm{Basic Maintenance}, \textrm{Critical Repair}\}$. 
            The transitions and rewards are defined in Table \ref{TransRew}. 
            Each transition matrix is Markov because, for all actions $i\in\mathcal{I}$, we have $\tau_1(P(i))<1$. Therefore, the ergodic property holds. 
            \begin{table}[ht!]
            \centering
            \caption{Transition and Reward Matrices for each Action: \textbf{G}, \textbf{F}, and \textbf{P} stand for \textit{Good}, \textit{Fair}, and \textit{Poor} condition, respectively.}
            \label{TransRew}
            \begin{adjustbox}{max width=\textwidth}
                \begin{tabular}{cccccccccccccc}
                    \toprule
                     & \multicolumn{4}{c}{\textbf{Wait}} & \multicolumn{4}{c}{\textbf{Basic Maintenance}} & \multicolumn{4}{c}{\textbf{Critical Repair}} \\
                    \cmidrule(lr){2-5} \cmidrule(lr){6-9} \cmidrule(lr){10-13}
                    & \textbf{G} & \textbf{F} & \textbf{P} & \textbf{Reward} & \textbf{G} & \textbf{F} & \textbf{P} & \textbf{Reward} & \textbf{G} & \textbf{F} & \textbf{P} & \textbf{Reward} \\
                    \midrule
                    \textbf{G} & \centering 0.9 & 0.1 & 0.0 & 0.9 & 0.95 & 0.05 & 0.0 & 0.1 & 1.0 & 0.0 & 0.0 & 0.1 \\
                    \textbf{F} & 0.0 & 0.7 & 0.3 & 0.55 & 0.8 & 0.2 & 0.0 & 0.7 & 0.9 & 0.1 & 0.0 & 0.5 \\
                    \textbf{P} & 0.0 & 0.1 & 0.9 & 0.05 & 0.0 & 0.3 & 0.7 & 0.4 & 0.3 & 0.65 & 0.05 & 0.85 \\
                    \bottomrule
                \end{tabular}
            \end{adjustbox}
            \end{table}
        \end{example}
            
\section{Proof of Results}\label{proofsresults}

    In Section \ref{Dec}, we consider the existence of the uniform value and the decidability of the approximation problem. 
    Next, we show the undecidability of the exact problem in Section \ref{Undec}. Finally, we prove that the uniform value is independent of the initial belief \ref{Addprop}. 

    \subsection{Proof of Theorem \ref{BMDPdec}}\label{Dec}
    
        Our approach exploits ergodicity to construct a finite-state stochastic game, referred to as the \textit{abstract stochastic game}, where the $N$-stage payoff deviates from that of the ergodic blind stochastic game by at most $\varepsilon$. This approach can be viewed as an aggregation scheme, where similar beliefs in the original game are grouped and represented by an ``abstract" belief in the finite-state stochastic game. A similar interpretation for MDPs is discussed in \cite{givan2000bounded}.\\
        
        We prove Theorem \ref{BMDPdec} in four steps. Observe that the belief update at stage $(n+1)$ after an action pair sequence $a^n=(a_1, \ldots, a_{n})$ can be expressed in ``matrix'' form by
            \begin{equation*}
                b_{n+1}^\top=b_1^\top T^n(a^n)
                    = b_1^\top P(a_1) \ldots P(a_n).
            \end{equation*}
    
        \noindent Let $P$ be a stochastic matrix and $b$ be a probability vector. We consider the following norms: $\lVert P \rVert_1\coloneqq \max_{k^\prime\in \mathcal{K}}\sum_{k\in \mathcal{K}} |p_{k,k^\prime}|$, $\lVert P\rVert_\infty \coloneqq \max_{k\in \mathcal{K}}\sum_{k^\prime\in \mathcal{K}}|p_{k,k^\prime}|$, and $\lVert b \rVert_1\coloneqq \sum_{k\in \mathcal{K}} |b(k)|$.
            
        \paragraph{Step 1.}\label{ConstAbMDP} We proceed to construct a finite-state stochastic game, termed \textit{abstract stochastic game}. Consider $\Gamma=(\mathcal{K},\mathcal{I},\mathcal{J},p,g)$ an ergodic blind stochastic game and $\varepsilon>0$. The first statement is a consequence of Proposition \ref{UB}.

            \begin{Proposition}
                \label{upperbound}
                Let $n_0\leq\tfrac{ 3^{|\mathcal{K}|}-2^{|\mathcal{K}|+1}+1}{2}$ such that \eqref{check} is satisfied for every action pair sequence $a^{n_0}$ and let $n_\varepsilon:=n_0 \left \lceil \tfrac{\ln(\varepsilon)}{\ln (\sup_{a^{n_0}}\tau_1(T^{n_0}(a^{n_0}))) } \right\rceil$.
                Then, we have that, for every action pair sequence $a^{n_1}=(a_1,...,a_{n_1})$ with $n_1\geq n_\varepsilon$, 
                \begin{equation}
                    \tau_1(T^{n_1}(a^{n_1}))\leq \varepsilon. \label{ineq2}
                \end{equation}
            \end{Proposition}

            \noindent When the context is clear, we denote $n_\varepsilon$ simply by $n$, omitting the dependence on $\varepsilon$.

            \begin{proof}[Proof of Proposition \ref{upperbound}]
                By Proposition \ref{UB}, a blind stochastic game is ergodic if and only if there exists $n_0\leq ( 3^{|\mathcal{K}|}-2^{|\mathcal{K}|+1} + 1) / 2$ such that for every action pair sequence $a^{n_0}$ we have
                \begin{equation*}
                    \tau_1(T^{n_0}(a^{n_0}))<1.
                \end{equation*}
                Consider the set of products of stochastic matrices $T^{n_0}(a^{n_0})$ where $a^{n_0}$ is every action pair sequence of length $n_0$. 
                Denote by $\overline{a}^{n_0}$ the sequence of action pairs of length $n_0$ that maximizes $\tau_1(T^{n_0}(a^{n_0}))$, i.e., $\overline{a}^{n_0}\coloneqq\textrm{argmax}_{a^{n_0}}\tau_1(T^{n_0}(a^{n_0}))$ and $\overline{\tau}(n_0)\coloneqq \tau_1(T^{n_0}(\overline{a}^{n_0}))$, with $\overline{\tau}(n_0)<1$. 
                For every $\varepsilon>0$ and taking $\overline{n}(n_0,\varepsilon) \coloneqq \left\lceil \ln(\varepsilon) / \ln (\overline{\tau}(n_0)) \right\rceil$, we have
                \begin{align*}
                    [\overline{\tau}(n_0)]^{\overline{n}(n_0,\varepsilon)}\leq \varepsilon.
                \end{align*}
                Moreover, for all action pair sequences of length $\overline{n}(n_0,\varepsilon)n_0$, it holds that
                \begin{equation*}
                    \tau_1\left(T^{\overline{n}(n_0,\varepsilon)n_0}\left(a^{\overline{n}(n_0,\varepsilon)n_0}\right)\right)\leq \left[\overline{\tau}(n_0)\right]^{\overline{n}(n_0,\varepsilon)},
                \end{equation*}
                by submultiplicativity of the coefficient of ergodicity $\tau_1$. Therefore, the result follows.
            \end{proof} 
            
            \noindent Given $\varepsilon>0$, define $\mathcal{T}(\varepsilon)$ as the finite set of forward products of transition matrices satisfying \eqref{ineq2} of length $n$, where $n$ is given by Proposition \ref{upperbound}.  In particular, we have that, for all $T^n(a^n) \in \mathcal{T}(\varepsilon)$, 
            \begin{equation*}
                \tau_1(T^n(a^n))\leq \varepsilon.
            \end{equation*}
    
            We construct an ``abstract" set of stable matrices, denoted by $\widetilde{\mathcal{T}}(\varepsilon)$, which approximates $\mathcal{T}(\varepsilon)$.
            Each matrix is approximated by another with equal rows corresponding to the average over each row. 
            Formally, for every action pair sequence $a^n\in (\mathcal{I}\times \mathcal{J})^n$, consider the matrix $T^n(a^n)\in \mathcal{T}(\varepsilon)$.
            We define $\widetilde{T}^n(a^n)\in \widetilde{\mathcal{T}}(\varepsilon)$ by
            \begin{equation*}
                \tilde{t}^n_{k,k^\prime}(a^n)\coloneqq\dfrac{1}{|\mathcal{K}|}\sum_{\overline{k}=1}^{|\mathcal{K}|} t^n_{\overline{k},k^\prime}(a^n).
            \end{equation*}
            Each stable matrix $\widetilde{T}^n\in \widetilde{\mathcal{T}}(\varepsilon)$ represents a unique belief state after $n$ stages. 
            Indeed, for every initial belief $b\in\Delta(\mathcal{K})$ and matrix $\widetilde{T}^n\in \widetilde{\mathcal{T}}(\varepsilon)$, the belief update is given by 
            \begin{align*}
                b^\prime(k^\prime) = b^\top (\widetilde{T}^n)_{k^\prime} = \sum_{k\in\mathcal{K}} b(k)\tilde{t}^n_{k,k^\prime},
            \end{align*} 
            for each $k^\prime\in\mathcal{K}$. By definition of stable matrices, the transition probabilities $\tilde{t}^n_{k, k'}$ are constant across all rows $k\in \mathcal{K}$ for each column $k^\prime\in \mathcal{K}$. Thus, the belief update becomes a convex combination of terms with equal values. Consequently, the belief update is independent of the initial belief and depends solely on the terms of the stable matrix $\widetilde{T}^n$.
            Therefore, the stable property of stochastic matrices is crucial for ensuring a finite state space in the abstract stochastic game.
    
            \textrm{}
            
            Consider $b_1\in\Delta(\mathcal{K})$ an initial belief and $\Gamma=(\mathcal{K},\mathcal{I},\mathcal{J},p,g)$ an ergodic blind stochastic game. Define the set of abstract beliefs by
            \begin{equation}
                    \mathcal{B}^\ast\coloneqq\left\{b^\ast\in \Delta(\mathcal{K}) \,|\, \exists \widetilde{T}^n\in \widetilde{\mathcal{T}}(\varepsilon)\textrm{ such that } b^\ast=b_1^\top\widetilde{T}^n\right\}\cup\{b_1\}.\label{beliefabs}
            \end{equation}
            Note that the set of abstract beliefs $\left\{b^\ast\in \Delta(\mathcal{K}) \,|\, \exists \widetilde{T}^n\in \widetilde{\mathcal{T}}(\varepsilon)\textrm{ such that } b^\ast=b_1^\top\widetilde{T}^n\right\}$ is independent of the initial belief $b_1\in \Delta(\mathcal{K})$. For $m\in [0..n-1]$, we will write $\mathcal{B}^\ast\times(\mathcal{I}\times \mathcal{J})^m\coloneqq \{(b^\ast,a_1,\ldots,a_m)\,|\,b^\ast\in \mathcal{B}^\ast\}$. 
            The abstract stochastic game, denoted by $\mathcal{G}^\ast(b_1,\varepsilon)$, is defined by a $5$-tuple $\mathcal{G}^\ast(b_1,\varepsilon)=(\mathcal{X},\mathcal{I},\mathcal{J},\overline{p}^\ast,\overline{g}^\ast)$, where:
            \begin{itemize}
                \item $\mathcal{X}$ is the finite set of states, defined by \begin{equation*}
                    \mathcal{X}\coloneqq \bigcup_{m=0}^{n-1}(\mathcal{B}^\ast\times(\mathcal{I}\times \mathcal{J})^m);
                \end{equation*}
                \item $\mathcal{I}$ and $\mathcal{J}$ are the finite sets of actions for Player $1$ and Player $2$, respectively;
                \item $\overline{p}^\ast \colon \mathcal{X}\times\mathcal{I}\times \mathcal{J}\to \mathcal{X}$ is the deterministic transition function that gives the successor state according to current state $x$ and action pair $(i,j)\in \mathcal{I}\times \mathcal{J}$; 
                \item $\overline{g}^\ast \colon \mathcal{X}\times \mathcal{I}\times \mathcal{J}\to[0,1]$ is the stage reward function.
            \end{itemize}
    
            Define $\text{proj}\colon\mathcal{X}\to\Delta(\mathcal{K})$ the function that assigns a belief state to each state of the abstract stochastic game. Given $x\in\mathcal{X}$, the function $\text{proj}$ is defined by
            \begin{equation*}
                \text{proj}(x)(k) \coloneqq
                \left\{
                    \begin{array}{ll}
                        b^\ast(k) & \mbox{ if } x=(b^\ast) \\
                        \sum_{\overline{k}\in \mathcal{K}} b^\ast(\overline{k}) t^m_{\overline{k},k}(a_1,...,a_m) & \mbox{ if }x=(b^\ast,a_1,...,a_m),
                    \end{array}
                \right.
            \end{equation*}  
            where $T^m(a^m)=T^m(a_1,...,a_m)=P(a_1)\ldots P(a_m)$ for $m\in [1..n-1]$.\\
    
            Given $x\in \mathcal{X}$, where $x$ is of the form $(b^\ast,a^{m})$ for $m\in[0..n-1]$, and an action pair $a\in\mathcal{I}\times \mathcal{J}$, define the \textit{abstract update} as
            \begin{equation*}
                \psi^\ast(x,a)\coloneqq 
                \left\{
                    \begin{array}{ll}
                        (b^\ast,a_1,\cdots,a_m,a) & \mbox{ if } m\in[0,n-2] \\
                        (\text{proj}(x)^\top\widetilde{T}^n(a^n)) & \mbox{ if }m=n-1,
                    \end{array}
                \right.
            \end{equation*}
            where $a^n=(a^{n-1},a)$. The \textit{abstract update} will compute the deterministic successor state $x^\prime\in\mathcal{X}$ given the current state $x$ and the action pair $a\in\mathcal{I}\times \mathcal{J}$. 
            
            Define the \textit{abstract transition function} as 
            \begin{align}
                \overline{p}^\ast(x^\prime|x,a)\coloneqq  
                \left\{
                \begin{array}{ll}
                    1 & \mbox{if } x^\prime=\psi^\ast(x,a) \label{transabs}\\
                    0 & \mbox{otherwise,}
                \end{array}
                \right.
            \end{align}
            where $x,x^\prime\in\mathcal{X}$ and $a\in\mathcal{I}\times \mathcal{J}$.\\
            
            For every state $x\in \mathcal{X}$ and action pair $(i,j)\in \mathcal{I}\times \mathcal{J}$, the \textit{abstract reward function} is defined by
            \begin{equation*}
                \overline{g}^\ast(x,i,j)
                    \coloneqq \sum_{k\in\mathcal{K}}\text{proj}(x)(k)\cdot g(k,i,j).
            \end{equation*}
            For each stage $m\geq 1$, let $\overline{G}^\ast_m\coloneqq \overline{g}^\ast(X_m,A_m)$ denote the stage reward function, where $X_m\in\mathcal{X}$ and $A_m\in\mathcal{I}\times \mathcal{J}$. Let $N\in \mathbb{N}^\ast$. The $N$-stage objective of the abstract stochastic game given by strategy pair $(\sigma,\pi)$ is defined by
            \begin{equation*}
                \gamma_{N,\varepsilon}^{b_1,\ast}(\sigma,\pi)\coloneqq \mathbb{E}_{\sigma,\pi}^{b_1}\left(\dfrac{1}{N}\sum_{m=1}^N \overline{G}^\ast_m\right),
            \end{equation*} 
            and the $N$-stage value, denoted by $v_{N,\varepsilon}^\ast(b_1)$, is defined by
            \begin{equation*}
                v_{N,\varepsilon}^\ast(b_1)\coloneqq\sup_{\sigma\in \Sigma}\inf_{\pi\in \Pi}\gamma_{N,\varepsilon}^{b_1,\ast}(\sigma,\pi)=\inf_{\pi\in \Pi}\sup_{\sigma\in \Sigma}\gamma_{N,\varepsilon}^{b_1,\ast}(\sigma,\pi).
            \end{equation*}
            
            \noindent Recall that for stochastic games with finite states and finite action sets, Mertens and Neyman proved that the uniform value exists \cite{mertens1981stochastic}.
            
            \begin{remark}
                Abstract stochastic games consist in a collection of stochastic games indexed by the initial belief of the original game.
            \end{remark}
                    
            \paragraph{Step 2.} We analyze the belief dynamics within an ergodic blind stochastic game and its corresponding abstract stochastic game, proving that they remain closely aligned.

            \begin{Lemma}\label{beliefPOMDPs}
                Let $b_1\in \Delta(\mathcal{K})$ be an initial belief, $\Gamma$ an ergodic blind stochastic game, and $\varepsilon>0$.
                For all $m\in\mathbb{N}^\ast$, strategy pairs $(\sigma,\pi) \in \Sigma\times\Pi$, and histories $h_{m}\in \mathcal{H}_m$, the states of the abstract stochastic game $\mathcal{G}^\ast(b_1,\varepsilon)$ satisfy
                \begin{equation*}
                    \left \| b_{m,\sigma,\pi}^{b_1,h_m} - \textnormal{proj}\left(x_{m,\sigma,\pi}^{b_1,h_m}\right) \right \|_1 \leq 4\varepsilon,
                \end{equation*}
                where $x_{m,\sigma,\pi}^{b_1,h_m}$ denotes the state of the abstract stochastic game at stage $m$, which is induced by the strategy pair $(\sigma,\pi)$, starting from the initial belief $b_1$, and conditioned on the realization of the history $h_m$. 
            \end{Lemma}
           
            \begin{proof}[Proof of Lemma \ref{beliefPOMDPs}]
                Recall that the abstract stochastic game $\mathcal{G}^\ast(b_1,\varepsilon)$ is constructed as follows:
                \begin{enumerate}
                    \item By Proposition \ref{upperbound}, for all $\varepsilon>0$, there exists an integer $n$ with the associated set of matrices $\mathcal{T}(\varepsilon) = \{ T^n(a^n) \}$ satisfying that, for all action pair sequences $a^n$, 
                    $$
                        \tau_1(T^n(a^n)) \leq \varepsilon \,,
                    $$
                    i.e., each matrix in $\mathcal{T}(\varepsilon)$ has similar rows;
    
                    \item Associated with $\mathcal{T}(\varepsilon)$, we construct the abstract set of stable matrices, denoted $\widetilde{\mathcal{T}}(\varepsilon)$;
    
                    \item Each matrix in $\widetilde{\mathcal{T}}(\varepsilon)$ can be regarded as a belief;
    
                    \item In the abstract stochastic game:
                    \begin{itemize}
                        \item Using $\text{proj}$, each state $x\in\mathcal{X}$ is related to a specific belief in $\Delta(\mathcal{K})$;
                        \item The sets of actions is the same as in the original ergodic blind stochastic game.
                    \end{itemize}
                \end{enumerate}
                
                \noindent Fix a strategy pair $(\sigma,\pi)\in \Sigma\times \Pi$. We prove that, for all $m \ge 1$ and $h_m\in \mathcal{H}_m$,
                \begin{equation*}
                    \left \|b_{m,\sigma,\pi}^{b_1,h_m}-\text{proj}\left(x^{b_1,h_m}_{m, \sigma,\pi}\right)\right \|_1 \leq 4\varepsilon \,.
                \end{equation*}
                We consider blocks of size $n$.
                Let $l \geq 0$, $a^n$ an action pair sequence, and $h_{(l+1)n+1}=(h_{ln+1},a^n)$ for $h_{ln+1}\in \mathcal{H}_{ln+1}$. We recall the following relations: 
                \begin{itemize}
                    \item $b_{(l+1)n+1,\sigma,\pi}^{b_1,h_{(l+1)n+1}} = b_{ln+1,\sigma,\pi}^{b_1,h_{ln+1}\top} T^n(a^n)$;
                    \item $\text{proj}\left(x^{b_1,h_{(l+1)n+1}}_{(l+1)n+1, \sigma,\pi}\right)= \text{proj}\left(x^{b_1,h_{ln+1}}_{ln+1,\sigma,\pi}\right)^\top \widetilde{T}^n(a^n)$. \\
                \end{itemize}
                    
                We prove the claim by an induction argument on $l\in \mathbb{N}$. For clarity, we omit the dependence of $b_{m,\sigma,\pi}^{b_1,h_m}$ and $\text{proj}\left(x_{m,\sigma,\pi}^{b_1,h_m}\right)$ on $\sigma,\pi$ and $h_m$.
                
                \paragraph{Base case} We start by observing that the base case holds, i.e., when $l=0$. By construction of the abstract stochastic game, for all $m\in [0..n-1]$ and action pair sequences $a^m$, we have that
                \begin{equation*}
                    \left \| b_{m+1}^{b_1}-\text{proj}\left(x^{b_1}_{m+1}\right) \right \|_1 = 0, 
                \end{equation*}
                where $b_{m+1}^{b_1}=b_1^\top T^m(a^m)$ and $x^{b_1}_{m+1}=(b_1,a^m)$. Moreover, we can see that, for every action pair sequences $a^n$,
                \begin{align*}
                    \left \| b^{b_1}_{n+1}- \text{proj}\left(x_{n+1}^{b_1}\right) \right \|_1
                        &\leq \sum_{k=1}^{|\mathcal{K}|} \sum_{j=1}^{|\mathcal{K}|} b_1(j) \left | t^n_{j,k}(a^n)-\tilde{t}^n_{j,k}(a^n) \right |  \\
                        &= \sum_{j=1}^{|\mathcal{K}|} b_1(j)\sum_{k=1}^{|\mathcal{K}|} \left | t^n_{j,k}(a^n)-\tilde{t}^n_{j,k}(a^n) \right | \\
                        &\leq \sum_{j=1}^{|\mathcal{K}|} b_1(j) \, 2 \tau_1(T^n) 
                            \tag*{(\textrm{Def. $\tau_1$ and $\widetilde{T}^n$})}\\
                        &\leq 2\varepsilon. \tag*{(Proposition \ref{upperbound})}
                \end{align*}
    
                \paragraph{Induction step} We now assume that the claim holds for block $l$ and prove it holds for the block $l+1$. For every action pair sequence $a^n$,
                \begin{align*}
                    &\left \| b_{(l+1)n+1}^{b_1}-\text{proj}\left(x^{b_1}_{(l+1)n+1}\right) \right \|_1 \\
                        &\qquad = \left \| b_{(l+1)n+1}^{b_1}-\text{proj}\left(x^{b_1}_{ln+1}\right)^\top T^n(a^n) +\text{proj}\left(x^{b_1}_{ln+1}\right)^\top T^n(a^n) -\text{proj}\left(x^{b_1}_{(l+1)n+1}\right)\right \|_1\\
                        &\qquad \leq \left \| b_{ln+1}^{b_1\top} T^n(a^n) - \text{proj}\left(x^{b_1}_{ln+1}\right)^\top T^n(a^n) \right \|_1 \\
                        &\qquad \qquad +\left \| \text{proj}\left(x^{b_1}_{ln+1}\right)^\top T^n(a^n)-\text{proj}\left(x^{b_1}_{ln+1}\right)^\top \widetilde{T}^n(a^n) \right \|_1 \\
                        &\qquad \leq 2 \varepsilon + \left \| \text{proj}\left(x^{b_1}_{ln+1}\right)^\top T^n(a^n) - \text{proj}\left(x^{b_1}_{ln+1}\right)^\top \widetilde{T}^n(a^n) \right \|_1 
                            \tag*{(\textrm{Proposition } \ref{upperbound})}\\
                        &\qquad \leq 2\varepsilon + \sum_{k=1}^{|\mathcal{K}|} \sum_{j=1}^{|\mathcal{K}|} \text{proj}\left(x^{b_1}_{ln+1}\right)(j) \left | t^n_{j,k}(a^n)-\tilde{t}^n_{j,k}(a^n) \right | \\
                        &\qquad = 2\varepsilon + \sum_{j=1}^{|\mathcal{K}|} \text{proj}\left(x^{b_1}_{ln+1}\right)(j)\sum_{k=1}^{|\mathcal{K}|} \left | t^n_{j,k}(a^n)-\tilde{t}^n_{j,k}(a^n) \right | \\
                        &\qquad \leq 2\varepsilon + \sum_{j=1}^{|\mathcal{K}|} \text{proj}\left(x^{b_1}_{ln+1}\right)(j) \, 2 \tau_1(T^n) 
                            \tag*{(\textrm{Def. $\tau_1$ and $\widetilde{T}^n$})}\\
                        &\qquad \leq 4\varepsilon,
                \end{align*}    
                where the last inequality follows by Proposition \ref{upperbound}.
                
                We now consider the difference between each belief inside block $l+1$. For every $m\in [1..n-1]$, denote $T^m (a^m)\coloneqq P(a_1)\ldots P(a_m)$ with $a_k=(i_k,j_k)\in \mathcal{I}\times \mathcal{J}$ corresponding to the $k$-th action pair in $a^m$. We obtain that
                \begin{align*}
                    \left \| b_{ln+1+m}^{b_1}- \text{proj}\left(x^{b_1}_{ln+1+m}\right) \right \|_1
                    &= \left \| \left(b_{ln+1}^{b_1}-\text{proj}\left(x^{b_1}_{ln+1}\right)\right)^\top T^m \right \|_1\\
                    &\leq \sum_{k\in\mathcal{K}} \left|\left(\left(b_{ln+1}^{b_1}-\text{proj}\left(x^{b_1}_{ln+1}\right)\right)^\top T^m \right)_k\right|\\
                    &\leq \left \| b_{ln+1}^{b_1}-\text{proj}\left(x^{b_1}_{ln+1}\right)\right \|_1\left \| T^m\right \|_\infty\\
                    &\leq 4\varepsilon,
                \end{align*}
                where the last inequality follows from the induction hypothesis. As a result, it follows from the induction argument that the claim holds.
            \end{proof}
        
            \paragraph{Step 3.} Knowing that beliefs remain close in the belief and abstract stochastic games, we prove that the difference between the average rewards are also close.
        
            \begin{Lemma}\label{cvgmPOMDPs}
                Let $b_1\in \Delta(\mathcal{K})$ be an initial belief, $\Gamma$ an ergodic blind stochastic game, and $\varepsilon>0$.
                For all $N\in \mathbb{N}^\ast$ and strategy pairs $(\sigma,\pi) \in \Sigma\times \Pi$, the reward of the abstract stochastic game $\mathcal{G}^\ast(b_1,\varepsilon)$ satisfies
                \begin{equation*}
                    \left|\mathbb{E}_{\sigma,\pi}^{b_1}\left(\dfrac{1}{N}\sum_{m=1}^{N}\overline{G}_m\right)-\mathbb{E}_{\sigma,\pi}^{b_1}\left(\dfrac{1}{N}\sum_{m=1}^{N}\overline{G}_m^\ast\right)\right|
                        \leq 4\varepsilon .
                \end{equation*}
            \end{Lemma}    

            \begin{proof}[Proof of Lemma \ref{cvgmPOMDPs}]
                Let $N\in\mathbb{N}^\ast$, $b_1\in\Delta(\mathcal{K})$ an initial belief, and $\Gamma$ an ergodic blind stochastic game. 
                We construct the abstract stochastic game $\mathcal{G}(b_1,\varepsilon)$ as defined above.
                Recall that, the belief transitions in both games are deterministic. 
                As a consequence, for every pair of strategies $(\sigma,\pi)\in \Sigma\times \Pi$ and history $h_N\in \mathcal{H}_N$, the probability of a history in $\Gamma(b_1)$ and $\mathcal{G}(b_1,\varepsilon)$ is given by 
                \begin{align*}
                    \mathbb{P}_{\sigma,\pi}^{b_1}(H_{N}=h_{N})=\prod_{m=1}^{N-1}\sigma(i_m|h_{m})\pi(j_m|h_{m}).
                \end{align*}
                Therefore, we have that, for all $N\in \mathbb{N}^\ast$ and strategy pairs $(\sigma,\pi)\in\Sigma\times \Pi$,
                \begin{align*}
                    &\left|\mathbb{E}_{\sigma,\pi}^{b_1}\left(\dfrac{1}{N}\sum_{m=1}^{N} \overline{G}_m \right) - \mathbb{E}_{\sigma,\pi}^{b_1} \left(\dfrac{1}{N}\sum_{m=1}^{N}\overline{G}^\ast_m \right)\right|\\
                    &\qquad=\left|\sum_{h_{N+1}\in \mathcal{H}_{N+1}}\mathbb{P}_{\sigma,\pi}^{b_1}(H_{N+1}=h_{N+1})\left(\dfrac{1}{N}\sum_{m=1}^{N} \overline{g}_m(b_m,i_m,j_m)-\dfrac{1}{N}\sum_{m=1}^{N}\overline{g}^\ast_m(x_m,i_m,j_m) \right)\right|\\
                    &\qquad\leq\sum_{h_{N+1}\in \mathcal{H}_{N+1}}\mathbb{P}_{\sigma,\pi}^{b_1}(H_{N+1}=h_{N+1})\left(\dfrac{1}{N}\sum_{m=1}^{N} \left|\overline{g}_m(b_m,i_m,j_m)-\overline{g}^\ast_m(x_m,i_m,j_m)\right| \right)\\
                    &\qquad\leq\sum_{h_{N+1}\in \mathcal{H}_{N+1}}\mathbb{P}_{\sigma,\pi}^{b_1}(H_{N+1}=h_{N+1})\left(\dfrac{1}{N}\sum_{m=1}^{N}\sum_{k\in\mathcal{K}} g(k,i_m,j_m)\left|b_{m}^{b_1}(k)-\text{proj}(x^{b_1}_{m})(k)\right| \right)\\
                    &\qquad\leq\sum_{h_{N+1}\in \mathcal{H}_{N+1}}\mathbb{P}_{\sigma,\pi}^{b_1}(H_{N+1}=h_{N+1})\left(\dfrac{1}{N}\sum_{m=1}^{N}\left \| b_{m}^{b_1}-\text{proj}(x^{b_1}_{m})\right \|_1 \right)\\
                    &\qquad\leq 4 \varepsilon,\tag*{\textrm{(By Lemma \ref{beliefPOMDPs})}}
                \end{align*}
                where, for every $m\in [1..N]$, each action pair $(i_m,j_m)$ naturally corresponds to the $m$-th action pair in $h_{N+1}$. 
            \end{proof}
    
            \paragraph{Step 4.} We now prove Theorem \ref{BMDPdec}. First, we show that every ergodic blind stochastic game has a uniform value. Next, we prove that the decision version of approximating the uniform value in this class of games is decidable.
        
            \begin{proof}[Proof of Theorem \ref{BMDPdec}]
    
                We prove each statement in turn.\\

                \noindent \underline{First statement}. 
                Let $b_1\in \Delta(\mathcal{K})$ be an initial belief, $\Gamma$ an ergodic blind stochastic game, and $\varepsilon>0$. Consider the abstract stochastic game $\mathcal{G}^\ast(b_1,\varepsilon)$ as constructed previously. By Lemma \ref{cvgmPOMDPs}, for all $N\in \mathbb{N}^\ast$, we have 
                \begin{equation} \label{eq:close_payoff}
                    \left|\mathbb{E}_{\sigma,\pi}^{b_1}\left(\dfrac{1}{N}\sum_{m=1}^{N}\overline{G}_m\right)-\mathbb{E}_{\sigma,\pi}^{b_1}\left(\dfrac{1}{N}\sum_{m=1}^{N}\overline{G}_m^\ast\right)\right|
                        \leq 4\varepsilon .
                \end{equation}
                The game $\mathcal{G}^\ast(b_1,\varepsilon)$ is a stochastic game with finite state space and actions sets, where states and actions are perfectly observed. 
                Consequently, it has a uniform value, denoted by $v^*(b_1,\varepsilon)$. 
                                
                Let $v(b_1)$ be an accumulation point of the sequence $\{v^\ast(b_1,\varepsilon)\}_{\varepsilon>0}$. Let us show that both players can guarantee uniformly $v(b_1)$ in $\Gamma(b_1)$. Let $\varepsilon^\prime>0$, there exists $\varepsilon \leq \varepsilon'$ such that 
                \begin{equation} \label{eq:close_value}
                \left|v^\ast(b_1,\varepsilon)-v(b_1)\right| \leq \varepsilon^\prime. 
                \end{equation}
                By definition of the uniform value, Player 1 has a strategy $\sigma^*$ in $\mathcal{G}^\ast(b_1,\varepsilon)$ such that for some $n_0 \geq 1$, for all $N \geq n_0$, for all $\pi\in \Pi$,
                   \begin{equation*}
                \mathbb{E}_{\sigma^*,\pi}^{b_1}\left(\dfrac{1}{N}\sum_{m=1}^{N}\overline{G}_m^\ast\right)
                \geq v^*(b_1,\varepsilon)- \varepsilon^\prime.
                \end{equation*}
                Combining with \eqref{eq:close_payoff} and \eqref{eq:close_value}, we get that, for all $N \geq n_0$,
                 \begin{equation*}
                    \mathbb{E}_{\sigma^*,\pi}^{b_1}\left(\dfrac{1}{N}\sum_{m=1}^{N}\overline{G}_m\right)
                 \geq v(b_1)-6\varepsilon^\prime.
                \end{equation*}
                We deduce that Player 1 can uniformly guarantee $v(b_1)$. 
                Reversing the roles of players shows that Player 2 can uniformly guarantee  $v(b_1)$ as well. 
                Hence, $\Gamma(b_1)$ has a uniform value, equal to $v(b_1)$. 
        
                \textrm{}
        
                \noindent \underline{Second statement}. 
                Let $\varepsilon>0$, $b_1\in\Delta(\mathcal{K})$ an initial belief, and $N\in\mathbb{N}^\ast$. 
                Denote by $v_N(b_1)$ the $N$-stage value of $\Gamma$ and $v_N^\ast(b_1,\varepsilon)$ the $N$-stage value of $\mathcal{G}^\ast(b_1,\varepsilon)$. 
                By Lemma \ref{cvgmPOMDPs}, for all $\varepsilon>0$ and $N\in \mathbb{N}^\ast$, we have
                \begin{equation*}
                \left|v_N^\ast(b_1,\varepsilon)-v_N(b_1)\right| \leq 4 \varepsilon.
                \end{equation*}
                Because $\Gamma$ has a uniform value $v$, we have that $(v_N(b_1)) \xrightarrow[]{N \to \infty} v(b_1)$. 
                Similarly, because $\mathcal{G}^\ast(b_1,\varepsilon)$ has a uniform value, we get that
                $(v_{N,\varepsilon}^\ast(b_1)) \xrightarrow[]{N \to \infty} v^\ast(b_1,\varepsilon)$. 
                Therefore, it follows that 
                \begin{equation*}
                    \left|v^\ast(b_1,\varepsilon)-v(b_1)\right| \leq 4 \varepsilon.
                \end{equation*}
                The number of states and actions of $\mathcal{G}^\ast(b_1,\varepsilon)$ depends only on $\varepsilon$ and the number of states and actions of $\Gamma$. 
                Moreover, the decision version of computing the uniform value is decidable for finite-state stochastic games, see \cite{oliu2021new} for recent algorithms. 
                It follows that the decision version of approximating the uniform value in ergodic blind stochastic games is a decidable problem.
            \end{proof}
            
            Our approximation scheme is detailed in Algorithm \ref{AMBMDP}.
        
            \begin{algorithm}
                \caption{Approximation Scheme for Ergodic Blind Stochastic Games\label{AMBMDP}}
                \begin{algorithmic}[1]
                    \State \textbf{Input:} Initial belief $b_1\in\Delta(\mathcal{K})$, Blind stochastic game $\Gamma=(\mathcal{K},\mathcal{I},\mathcal{J},p,g)$, $\varepsilon>0$
                    \State Verify ergodic property of $\Gamma$
                    \State Construct the set of matrices $\mathcal{T}(\varepsilon)$ 
                    \State Derive the abstract set of matrices $\widetilde{\mathcal{T}}(\varepsilon)$ from $\mathcal{T}(\varepsilon)$ 
                    \State Construct the abstract stochastic game $\mathcal{G}^\ast(b_1,\varepsilon)$ 
                    \State Compute the uniform value $v^\ast(b_1,\varepsilon)$ of the abstract stochastic game $\mathcal{G}^\ast(b_1,\varepsilon)$
                    \State \textbf{Output:} Uniform value $v^\ast(b_1,\varepsilon)$ if $\Gamma$ is ergodic
                \end{algorithmic}  
            \end{algorithm}

    \subsection{Proof of Theorem \ref{EVU}}\label{Undec}
        
        To approach the exact problem, we consider probabilistic finite automata (PFAs) \cite{madani2003undecidability}. While PFAs are models that accept or reject strings, blind MDPs are used for solving stochastic sequential optimization problems where the decision-maker receives no information of the system. These models are tightly connected. Indeed, the alphabet in PFAs corresponds to actions in blind MDPs. Further, the notion of acceptance in PFAs corresponds to a reachability objective in blind MDPs. Due to this connection, undecidability results in PFAs also hold for blind MDPs. 
        
        \begin{Definition}[PFA \cite{chatterjee2010probabilistic}]
            A PFA, denoted by $\mathcal{M}$, is defined as a $5$-tuple $\mathcal{M}=(\mathcal{K}, \mathcal{B}, \mathcal{I}, p, \delta_{k_1})$, where: 
            \begin{itemize}
                \item $\mathcal{K}$ is the finite set of states; 
                \item $\mathcal{I}$ is the finite set of symbols; 
                \item $p \colon \mathcal{K}\times \mathcal{I}\to\Delta(\mathcal{K})$ is the probability distribution over the successor state given the current state $k\in\mathcal{K}$ and the symbol $i\in \mathcal{I}$; 
                \item $\mathcal{B}\subseteq \mathcal{K}$ is the set of nonabsorbing accepting states, where a state $k$ is absorbing if $p(k|k,i)=1$ for all $i\in \mathcal{I}$;  
                \item $\delta_{k_1}$ is the initial belief.
            \end{itemize}
        \end{Definition}
    
        \noindent A blind MDP is defined similarly as a PFA \cite{chatterjee2022finite}. Indeed, a blind MDP, denoted by $\Gamma$, is defined by a $4$-tuple $\Gamma=(\mathcal{K},\mathcal{I},p,g)$, where: $\mathcal{K}$ is the finite set of states and $\mathcal{I}$ is the finite set of actions; $p \colon \mathcal{K}\times \mathcal{I}\rightarrow\Delta(\mathcal{K})$, represented as $p(k^\prime|k,i)$, is the probabilistic transition function that gives the probability distribution over the successor states given a state $k\in \mathcal{K}$ and an action $i\in \mathcal{I}$. We represent by $P(i)$ the transition matrix for each action $i\in \mathcal{I}$; $g \colon \mathcal{K}\times\mathcal{I}\rightarrow[0,1]$ is the stage reward function; $b_1\in \Delta(\mathcal{K})$ is the \textit{initial belief}, which represents the initial probability distribution over the state space. Recall that a Markov blind MDP is such that $P(i)$ is Markov for every $i\in \mathcal{I}$. Starting from initial belief $b_1\in\Delta(\mathcal{K})$, the blind MDP is denoted by $\Gamma(b_1)$.
        
        \textrm{}

        Let $\mathbb{P}^{\delta_{k_1}}_w(K_{N+1}\in \mathcal{B})$ be the probability of acceptance of a word $w\in \mathcal{I}^N$, where $\mathbb{P}_w^{\delta_{k_1}}$ denotes the measure over $\mathcal{K}^{N+1}$ induced by $w$ when the automaton starts with initial belief $\delta_{k_1}$.
        Consider the universality problem for PFAs from \cite[Theorem 3.2]{madani2003undecidability}.
            
        \begin{Theorem}
            \label{Result: PFA undecidability}
            Given a PFA, deciding whether there exists a word with acceptance probability strictly greater than $1/2$ is undecidable.
        \end{Theorem}
    
        \noindent \textit{Proof Sketch.} We reduce the universality problem for PFAs to the decision problem of computing the uniform value for Markov blind MDPs.
        More formally, given a PFA $\mathcal{M}=(\mathcal{K},\mathcal{I},p,\delta_{k_1})$, we construct a Markov blind MDP $\Gamma$ as follows:
        \begin{itemize}
            \item Add a state $\hat{k}$ such that each state $k\in\mathcal{K}$ reaches $\hat{k}$ with a positive probability. Consequently, every transition matrix of $\Gamma$ is Markov. 
            \item Introduce an action \texttt{Restart} that sends any state (including $\hat{k}$) deterministically to the PFA's initial state.
        \end{itemize}
        We then prove that the value of the PFA is strictly larger than $1/2$ if and only the long-run average value of the blind MDP is strictly larger than $1/2$. Because the universality problem for PFAs is undecidable by \Cref{Result: PFA undecidability}, this equivalence transfers undecidability to the exact computation of the uniform value in Markov blind MDPs.
        
        \begin{proof}[Proof of Theorem \ref{EVU}]
            Consider a PFA $\mathcal{M}=(\mathcal{K},\mathcal{B},\mathcal{I},p,\delta_{k_1})$, where $\delta_{k_1}$ is the initial belief and $\mathcal{B}\subseteq \mathcal{K}$ the set of accepting states. 
            Given $\mathcal{M}$ and $\theta \in (0,1)$, we construct a blind MDP $\Gamma=(\mathcal{K}^\prime,\mathcal{I}^\prime,p^\prime,g)$ with $b_1=\delta_{k_1}$ as initial belief and the long-run average objective as follows:
            \begin{itemize}
                \item $\mathcal{K}^\prime=\mathcal{K}\cup \{\hat{k}\}$ is the finite set of states;
                \item $\mathcal{I}^\prime=\mathcal{I}\cup \{\texttt{Restart}\}$ is the finite set of actions;
                \item $p^\prime\colon\mathcal{K}^\prime\times \mathcal{I}^\prime\to\Delta(\mathcal{K}^\prime)$ is the probabilistic transition function, defined as follows:
                \begin{itemize}
                    \item For every action $i\in \mathcal{I}$ and state $k\in \mathcal{K}$, $p^\prime(k,i)(\hat{k})=\theta$, $p^\prime(k,i)(k^\prime)=(1-\theta)p(k,i)(k^\prime)$ for all states $k^\prime\in \mathcal{K}$, $p^\prime(\hat{k},i)(\hat{k})=1$;
                    \item When $i=\texttt{Restart}$, $p^\prime (k,\texttt{Restart})(k_1)=1$ for all states $k\in \mathcal{K}^\prime$.
                \end{itemize}
                \item $g\colon\mathcal{K}^\prime\times \mathcal{I}^\prime\to [0,1]$ is the stage reward function, defined as follows:
                \begin{itemize}
                    \item For every action $i\in \mathcal{I}$, we have $g(k',i)=1/2$ for all states $k' \in \mathcal{K'}$; 
                    \item When $i=\texttt{Restart}$, we have $g(\hat{k},\texttt{Restart})=1/2$, $g(k,\texttt{Restart})=1$ for all $k\in \mathcal{B}$, and $g(k, \texttt{Restart})=0$ for all $k\in \mathcal{K}\setminus\mathcal{B}$.
                \end{itemize}
            \end{itemize}
            
            In the blind MDP $\Gamma$, the set of accepting states $\mathcal{B}\subseteq \mathcal{K}^\prime$ is transient. 
            Note that, for all actions $i\in \mathcal{I}^\prime$, the transition matrix $P(i)$ is Markov. Therefore, the blind MDP $\Gamma$ is Markov which implies that $\Gamma$ is ergodic.
                
            Let us show that the acceptance probability of the PFA $\mathcal{M}$ is strictly greater than $1/2$ if and only if the long-run average value of the blind MDP is strictly greater than $1/2$.       
                
            Consider a PFA $\mathcal{M}$ where there exists a word $w$ of length $|w|=N$ that has an acceptance probability strictly greater than $1/2$. 
            We provide a strategy in the blind MDP $\mathcal{G}$ that guarantees a payoff strictly greater than $1/2$.  
                
            The strategy consists in repeatedly playing the actions in $w$ followed by $\texttt{Restart}$. 
            To compute the payoff it guarantees, focus on a single block. 
            Denote the strategy $\sigma$ and note that
            \begin{align*}
                &\mathbb{E}^{\delta_{k_1}}_\sigma \left( \dfrac{1}{N+1}\sum_{m=1}^{N+1} G_m \right)\\
                    &\qquad=\dfrac{1}{N+1}\left[\dfrac{N}{2}+ \mathbb{P}_w^{\delta_{k_1}}(K_{N+1}=\hat{k})\dfrac{1}{2}+\left( 1 - \mathbb{P}_w^{\delta_{k_1}}(K_{N+1}=\hat{k}) \right) \mathbb{P}_w^{\delta_{k_1}}(K_{N+1}\in \mathcal{B})\right]\\
                    &\qquad>\dfrac{1}{N+1}\left[\dfrac{N}{2}+\mathbb{P}_w^{\delta_{k_1}}(K_{N+1}=\hat{k})\dfrac{1}{2} + \left( 1 - \mathbb{P}_w^{\delta_{k_1}}(K_{N+1}=\hat{k}) \right) \dfrac{1}{2}\right]\\
                    &\qquad =1/2.
            \end{align*}
            At stage $N+2$, the state is again $k_1$. By repeating the argument on each block of size $N+1$, we obtain that this strategy guarantees strictly more than $1/2$. 
        
            Consider now that the blind MDP $\mathcal{G}$ has a long-run average value strictly greater than $1/2$.
            We show that the PFA $\mathcal{M}$ has a word with an acceptance probability strictly greater than $1/2$.
            By \cite{chatterjee2022finite}, there exists an eventually periodic strategy that guarantees strictly more than $1/2$. We claim that such a strategy should play $\texttt{Restart}$ an infinite number of times. Indeed, otherwise the game would remain in $\hat{k}$ from some stage, and the strategy would achieve payoff $1/2$, which is a contradiction.
            Since the strategy is eventually cyclic and after playing once the action $\texttt{Restart}$ the state moves to $k_1$ with probability $1$, we may consider that the strategy repeats a cycle of the form $(i_1, i_2, \ldots, i_N, \texttt{Restart})$. 
            Moreover, if there exists $m \in [1..N]$ such that $i_m = \texttt{Restart}$, the payoff the strategy guarantees is a weighted average between the payoffs that the two strategies given by the cycles $(i_1, \ldots, i_{m-1}, \texttt{Restart})$ and $(i_{m+1}, \ldots, i_{N}, \texttt{Restart})$ guarantee. 
            Therefore, repeating this argument, we may consider a strategy that repeats a cycle of the form $(i_1, i_2, \ldots, i_N, \texttt{Restart})$, where $i_m \in \mathcal{I}$ for all $m \in [1..N]$. 
            We prove that the word $w = i_1 i_2\ldots i_N$ is accepted by the PFA $\mathcal{M}$ with probability strictly larger than $1/2$.
        
            Indeed, denote the strategy $\sigma$ and note that
            \begin{align*}
                \frac{1}{2} &< \mathbb{E}^{\delta_{k_1}}_\sigma \left( \dfrac{1}{N+1}\sum_{m=1}^{N+1} G_m \right)\\
                & = (1 - \theta)^N \mathbb{P}_w^{\delta_{k_1}}\left(K_{N+1} \in \mathcal{B}\right) + \left(1 - (1 - \theta)^N\right) \dfrac{1}{2}.
            \end{align*}
            Therefore, $\mathbb{P}_w^{\delta_{k_1}}(K_{N+1} \in \mathcal{B}) > 1/2$, i.e., $w$ is accepted by $\mathcal{M}$ with probability strictly larger than $1/2$.
            By Theorem \ref{Result: PFA undecidability}, it follows that the decision version of computing the long-run average value for the class of Markov blind MDPs is undecidable. 
        \end{proof}

    \subsection{Proof of Theorem \ref{0opt}}\label{Addprop}

        In this section, we prove that the uniform value is independent of the initial belief.

        \begin{proof}[Proof of Theorem \ref{0opt}]
            Consider an ergodic blind stochastic game $\Gamma$ and arbitrary initial beliefs $b_1, b_1^\prime \in \Delta(\mathcal{K})$.
            We show that the uniform values of $\Gamma(b_1)$ and $\Gamma(b_1^\prime)$ are equal.

            By Theorem \ref{BMDPdec}, the uniform value exists in $\Gamma(b_1)$ and $\Gamma(b_1^\prime)$. 
            Consider an arbitrary $\varepsilon > 0$. 
            We construct the abstract stochastic games $\mathcal{G}^\ast(b_1,\varepsilon)$ and $\mathcal{G}^\ast(b_1^\prime,\varepsilon)$ as explained in \Cref{Dec}. 
            Consider an arbitrary strategy pair $(\sigma, \pi) \in \Sigma \times \Pi$.
            Note that it induces the same state processes, except possibly during the first $n_\varepsilon$ stages.
            Hence, for every horizon $N \geq n_\varepsilon$,
            \begin{align*}
                \left|\mathbb{E}^{b_1}_{\sigma,\pi}\left(\dfrac{1}{N}\sum_{m=1}^N\overline{G}_m^\ast\right)-\mathbb{E}^{b_1^\prime}_{\sigma,\pi}\left(\dfrac{1}{N}\sum_{m=1}^N\overline{G}_m^\ast\right)\right|\leq \dfrac{n_\varepsilon}{N}.
            \end{align*}
            Therefore, $|v^\ast_N(b_1,\varepsilon) - v^\ast_N(b_1^\prime,\varepsilon)| \leq \tfrac{n_\varepsilon}{N}$. 
            Because the uniform value in the abstract stochastic game exists \cite{mertens1981stochastic}, taking $N \to \infty$ we deduce that, 
            \begin{equation*}
                v^\ast(b_1,\varepsilon) = v^\ast(b_1^\prime,\varepsilon).
            \end{equation*}
            Finally, by construction of the abstract games, 
            \begin{align*}
                |v(b_1) - v(b_1^\prime)|
                    &\leq |v(b_1)-v^\ast(b_1,\varepsilon)| + |v^\ast(b_1,\varepsilon)-v(b_1^\prime)|\\
                    &= |v(b_1)-v^\ast(b_1,\varepsilon)| + |v^\ast(b_1^\prime,\varepsilon)-v(b_1^\prime)|\\
                    &\leq 8 \varepsilon.
            \end{align*}
            The statement follows because $b_1$, $b_1^\prime$, and  $\varepsilon > 0 $ are arbitrary. 
        \end{proof}

\section{Discussion}\label{ExtPOMDPs}

        Hidden stochastic games \cite{renault2020hidden} extend blind stochastic games by allowing players to receive signals after taking actions. This model is known by various names in the literature, including stochastic games with signals \cite{solan2016stochastic} and partially observable stochastic games \cite{hansen2004dynamic}. The term ``Partially observable stochastic game" is especially prevalent in formal method problems \cite{chatterjee2014partial}. 

        \begin{Definition} 
            A hidden stochastic game, denoted by $\Gamma$, is defined by a $7$-tuple $\Gamma=(\mathcal{K},\mathcal{I},\mathcal{J},\mathcal{S},p,q,g)$, where:
            \begin{itemize}
                \item $\mathcal{K}$ is the finite set of states;
                \item $\mathcal{I}$ and $\mathcal{J}$ are the finite sets of actions for Player 1 and Player 2, respectively;
                \item $\mathcal{S}$ is the finite set of signals;
                \item $p \colon \mathcal{K}\times \mathcal{I}\times \mathcal{J}\rightarrow\Delta(\mathcal{K})$, represented as $p(k^\prime|k,i,j)$, is the probabilistic transition function that gives the probability distribution over the successor states given a state $k\in \mathcal{K}$ and action pair $(i,j)\in \mathcal{I}\times \mathcal{J}$. We represent by $P(i,j)$ the transition matrix.
                \item $q \colon \mathcal{K}\times \mathcal{I}\times \mathcal{J}\rightarrow\Delta(\mathcal{S})$, expressed as $q(s|k,i,j)$, is the probabilistic observation function that gives the probability distribution over the observations given a state $k\in \mathcal{K}$ and action pair $(i,j)\in \mathcal{I}\times \mathcal{J}$. The observation matrix corresponding to any action pair $(i,j)\in\mathcal{I}\times\mathcal{J}$ is denoted by $Q(i,j)$; 
                \item $g \colon \mathcal{K}\times\mathcal{I}\times \mathcal{J}\rightarrow[0,1]$ is the stage reward function.
            \end{itemize}
        \end{Definition}
        
        For every action pair $(i,j)\in\mathcal{I}\times \mathcal{J}$ and signal $s\in\mathcal{S}$, define the $|\mathcal{K}|\times |\mathcal{K}|$ matrix $R(i,j,s)$ such that $r_{k,k^\prime}(i,j,s)\coloneqq q(s|k^\prime,i,j)\times p(k^\prime|k,i,j)$ for all $k,k^\prime$. Let $\mathcal{R}$ denote the set of sub-stochastic matrices such that $\mathcal{R}\coloneqq\left\{R(i,j,s)\, | \,(i,j)\in\mathcal{I}\times\mathcal{J},s\in\mathcal{S}\right\}$. 
        We say that a sub-stochastic matrix $R$ is scrambling if for every two rows, there exists a common successor, i.e., a column with a positive entry in both rows. 
        Then, we can generalize the class of scrambling blind stochastic games to hidden stochastic games as follows.
                
        \begin{Definition}
            A hidden stochastic game is scrambling if every matrix in $\mathcal{R}$ is scrambling.
        \end{Definition} 
        \noindent Because Markov blind stochastic games form a subclass of scrambling hidden stochastic games, Theorem \ref{EVU} establishes that computing the uniform value of scrambling hidden stochastic games, if it exists, is undecidable. This naturally raises two questions: whether the uniform value actually exists and whether approximating it could be a decidable problem. These questions remain an open avenue for exploration.

        \textrm{} 

        We outline the main obstacles in extending methods for ergodic blind stochastic games to hidden stochastic games. Recall that the main approach for both blind stochastic games and hidden stochastic games involves considering their corresponding belief stochastic games. In blind stochastic games, transitions in the belief game are deterministic, whereas in hidden stochastic games, they are stochastic, influenced by random signals. The proof of decidability for ergodic blind stochastic games depends on a ``perfect coupling" between the ergodic blind stochastic game and the abstract stochastic game, enabled by deterministic transitions. This property facilitates a straightforward comparison between the two games. However, in hidden stochastic games, the stochastic transitions could lead to diverging paths between the hidden and abstract games, and thus an error propagation between the two, as previously noted by Rosenberg et al. \cite{rosenberg2002blackwell}. Finally, proving undecidability in scrambling hidden stochastic games could be challenging, as existing undecidability results typically rely on the blind single-player case.\\

        To conclude, our results are summarized in Table \ref{summary}.
        
        \begin{table}[ht!]
            \centering
            \renewcommand{\arraystretch}{1.2}
            \setlength{\tabcolsep}{4pt} 
            \small 
            \begin{tabular}{lccccc}
            \toprule
            \textbf{Class} &\textbf{ Uniform Value }&\textbf{ Exact Prob.} & \textbf{Approx. Prob.} & \textbf{Constant} &\textbf{ Suff. Cond.} \\
            \midrule
            Ergodic Blind MDPs & Yes & \textbf{Undec.} & \textbf{Dec.}  & \textbf{Yes} & \textbf{Yes} \\ 
            Ergodic Blind SGs & \textbf{Yes} & \textbf{Undec.} & \textbf{Dec.} & \textbf{Yes} & \textbf{Yes} \\
            Scrambling Hidden SGs & ? & \textbf{Undec.} & ? & ? & \textbf{Yes} \\
            \bottomrule
            \end{tabular}
            \caption{A question mark (?) denotes remaining open problems, while the new results presented in this paper are highlighted in bold. The existence of the uniform value for blind MDPs was established by Rosenberg et al. \cite{rosenberg2002blackwell}.}
            \label{summary}
        \end{table}

\section*{Acknowledgements} 

    This material is based upon work supported by the ANRT under the French CIFRE Ph.D program, in collaboration between NyxAir (France) and Paris-Dauphine University (Contract: CIFRE N° 2022/0513), by the French Agence Nationale de la Recherche (ANR) under reference ANR-21-CE40-0020 (CONVERGENCE project) and ANR-17-EURE-0010 (Investissements d’Avenir program), and partially supported by the ERC CoG 863818 (ForM-SMArt) grant and the Austrian Science Fund (FWF) 10.55776/COE12 grant. Part of this work was done at NyxAir (France) by David Lurie. Part of this work was done during a 1-year visit of Bruno Ziliotto to the Center for Mathematical Modeling (CMM) at University of Chile in 2023, under the IRL program of CNRS. The authors thank Marco Scarsini for helpful comments.

\bibliographystyle{alpha}
\bibliography{Biblio} 

\end{document}